\documentclass[11pt,oneside]{amsart}
\usepackage[utf8]{inputenc}
\usepackage[a4paper]{geometry}
\usepackage{fancyhdr}
\usepackage{float}
\usepackage{amsthm}
\usepackage{amstext}
\usepackage{amssymb}

\makeatletter

\providecommand{\tabularnewline}{\\}

\theoremstyle{plain}
\newtheorem{thm}{\protect\theoremname}
  \theoremstyle{plain}
  \newtheorem{lem}[thm]{\protect\lemmaname}
  \theoremstyle{plain}
  \newtheorem{cor}[thm]{\protect\corollaryname}
  \theoremstyle{plain}
  \newtheorem{prop}[thm]{\protect\propositionname}
  \theoremstyle{remark}
  \newtheorem{rem}[thm]{\protect\remarkname}

\makeatother

  \providecommand{\corollaryname}{Corollary}
  \providecommand{\lemmaname}{Lemma}
  \providecommand{\propositionname}{Proposition}
  \providecommand{\remarkname}{Remark}
\providecommand{\theoremname}{Theorem}

\begin{document}
\global\long\def\F{\mathcal{F}}
 \global\long\def\G{\mathcal{G}}
 \global\long\def\H{\mathcal{H}}
 \global\long\def\V{\mathcal{V}}
 \global\long\def\U{\mathcal{U}}
 \global\long\def\S{\mathcal{S}}
 \global\long\def\W{\mathcal{W}}
 \global\long\def\D{\mathcal{D}}
\global\long\def\J{\mathcal{J}}
 \global\long\def\filter{\mathbb{F}}
 \global\long\def\lm{\mathrm{lim}}
 \global\long\def\adh{\mathrm{adh}}
 \global\long\def\N{\mathcal{N}}
 \global\long\def\A{\mathcal{A}}
 \global\long\def\then{\Longrightarrow}
 \global\long\def\Adh{\mathrm{Adh}}
\global\long\def\Base{\mathrm{Base}}

\title{measure of compactness for filters in product spaces: Kuratowski-Mròwka
in CAP revisited}

\author{Frédéric Mynard and William Trott}

\address{F. Mynard, Department of Mathematics, New Jersey City University,
2039 John F. Kennedy Boulevard Jersey City, New Jersey 07305 }

\address{W. Trott, University of California Riverside}

\date{\today}
\begin{abstract}
The first author introduced a measure of compactness for families
of sets, relative to a class of filters, in the context of convergence
approach spaces. We characterize a variety of maps (types of quotient
maps, closed maps, and variants of perfect maps) as those respecting
this measure of compactness under one form or another. We establish
a product theorem for measure of compactness that yields as instances
new product theorems for spaces and maps, and new product characterizations
of spaces and maps, thus extending existing results from the category
of convergence spaces to that of convergence approach spaces. In particular,
results of the Mròwka-Kuratowski type are obtained, shedding new light
on existing results for approach spaces.
\end{abstract}



\maketitle

\section{Introduction}

R. Lowen introduced \cite{Lowen89} Approach Spaces as a powerful
tool bridging the gap between metric, topological and uniform spaces.
In that setting, many concepts have been unified via the introduction
of measures (e.g., measures of connectedness, measure of compactness
and of its variants \cite{baekelandlowen,baekelandlowen.LS,lowen.KMC}).
Consequently, as the category $\mathbf{Ap}$ of approach spaces became
a natural object of study, so did its quasitopos, extensional \cite{Lowen88,lowe88},
and cartesian closed \cite{nauwelaerts.thesis,cchullofAP} hulls.
In particular, the category $\mathbf{Cap}$ of convergence approach
spaces, which contains both $\mathbf{Ap}$ and the category $\mathbf{Conv}$
of convergence spaces both reflectively and coreflectively emerged
as a convenient setting for integrating metric-like and topological-like
studies. 

In \cite{mynardmeasureCAP}, the first author introduced a general
measure of compactness in $\mathbf{Cap}$ relative to a class of filters
that applied to filters to the effect that all known measures of compactness-like
properties for approach spaces (as in \cite{baekelandlowen,baekelandlowen.LS,lowen.KMC}),
as well as limit functions for various important reflections, were
recovered as instances. The purpose of the present paper is to study
its preservation under maps and finite products. As a consequence,
new results on measures of compactness and its variants for product
of spaces are obtained, and characterizations ``à la Mròwka-Kuratowski''
are obtained for a range of spaces and maps, including perfect maps
and its variants, and various kinds of quotient maps. Maybe more importantly,
all these results appear as instances of a single unifying principle.
In particular, this viewpoint sheds new light on the results of \cite{lowen.Mrowka},
which are among those revisited here.

\subsection{Terminology: convergence approach spaces and its subcategories}

Let $\left(\mathbb{F}X,\leq\right)$ denote the set of filters on
$X,$ ordered by inclusion (inverse to the monad order). Let $\mathbb{U}X$
be the subset of $\mathbb{F}X$ formed by ultrafilters on $X$ and,
given $\mathcal{G}\in\mathbb{F}X$, let $\mathbb{U}(\mathcal{G})$
be the set of ultrafilters that are finer than $\mathcal{G}.$ If
$A\subset X$, then $A^{\uparrow}:=\{B\subset X:A\subseteq B\}$ and
if $\A\subset2^{X}$ then $\A^{\uparrow}:=\bigcup_{A\in\A}A^{\uparrow}$.

Following \cite{Lowen88} and \cite{lowe88}, we call \emph{convergence-approach
limit} on $X$ a map $\lambda:\mathbb{F}X\rightarrow[0,\infty]^{X}$
which fulfills the properties:
\begin{gather}
\forall x\in X,\,\lambda(\{x\}^{\uparrow})(x)=0;\tag{CAL1}\label{eq:CAL1}\\
\G\geq\F\then\lambda(\F)\geq\lambda(\G);\tag{CAL2}\label{eq:CAL2}\\
\forall\F,\G\in\mathbb{F}X,\,\lambda(\F\wedge\G)=\lambda(\F)\vee\lambda(\G),\tag{CAL3}\label{eq:CAL3}
\end{gather}
 where of course (\ref{eq:CAL2}) follows from (\ref{eq:CAL3}) and
is therefore redundant.

$(X,\lambda)$, shortly $X$, is called a \emph{convergence-approach
space}. A map $f:X\rightarrow Y$ between two convergence-approach
spaces is a \emph{contraction} if 
\[
\lambda_{Y}(f(\mathcal{F}))~(f(\cdot))\leq\lambda_{X}(\mathcal{F})(\cdot),
\]
for every $\mathcal{F}\in\mathbb{F}X$. The category with convergence-approach
spaces as objects and contractions as morphisms is a cartesian-closed
topological category denoted \textbf{Cap} \cite{Lowen88}. Each convergence
space $X$ can be considered as a convergence-approach space by stating
\[
\lambda_{X}(\mathcal{F})(x)=\begin{cases}
0 & \text{if }x\in\lm_{X}\F\\
\infty & \text{otherwise.}
\end{cases}
\]
Moreover, the category \textbf{Conv} of convergence spaces (and continuous
maps) is included both reflectively and coreflectively in \textbf{Cap}.
Indeed, if $\lambda$ is a convergence-approach, then its \textbf{Conv}-coreflection
is $c(\lambda)$ defined by $x\in\lim_{c(\lambda)}\mathcal{F}$ if
and only if $\lambda(\mathcal{F})(x)=0$, while its \textbf{Conv}-reflection
is $r(\lambda)$ defined by $x\in\lim_{r(\lambda)}\mathcal{F}$ if
and only if $\lambda(\mathcal{F})(x)<\infty$. 

A convergence-approach $\lambda$ is a \emph{pseudo-approach space}
\cite{lowe88} if 
\begin{equation}
\forall\mathcal{F}\in\mathbb{F}X,~\lambda(\mathcal{F})=\underset{\mathcal{U}\in\mathbb{U}(\mathcal{F})}{\bigvee}\lambda(\mathcal{U});\tag{PSAP}\label{eq:PSAP}
\end{equation}
and it is a \emph{pre-approach space} \cite{Lowen88} if (\ref{eq:CAL3})
is strengthened to 
\begin{equation}
\lambda(\underset{j\in J}{\bigwedge}\mathcal{F}_{j})=\underset{j\in J}{\bigvee}\lambda(\mathcal{F}_{j}),\text{ for any family }(\mathcal{F}_{j})_{j\in J}\text{ of filters.}\tag{PRAP}\label{eq:PRAPdef}
\end{equation}
The category \textbf{Psap} of pseudo-approach spaces (and contractions)
contains the category \textbf{Pstop} of pseudotopological spaces (and
continuous maps) and the category \textbf{Prap} of pre-approach spaces
contains the category \textbf{Prtop} of pretopological spaces both
reflectively and coreflectively (via the restrictions of $c$ and
$r$). Both \textbf{Psap }and \textbf{Prap }are reflective subcategories
of \textbf{Cap}.

We say that two families $\mathcal{A}$ and $\mathcal{B}$ of subsets
of a set $X$ \emph{mesh}, in symbol $\mathcal{A}\#\mathcal{B},$
if $A\cap B\neq\emptyset$ whenever $A\in\mathcal{A}$ and $B\in\mathcal{B}.$
We write $A\#\mathcal{B}$ for $\{A\}\#\mathcal{B}.$ The \emph{grill
of a family }$\mathcal{A}$ of subsets of $X$ is $\mathcal{A}^{\#}=\{H\subset X:H\#\mathcal{A}\}.$ 

An \emph{approach space} is a pre-approach space fulfilling 
\begin{equation}
\lambda(\bigcup_{F\in\F}\bigcap_{x\in F}\G(x))(\cdot)\leq\lambda(\F)(\cdot)+\underset{x\in X}{\bigvee}\lambda(\mathcal{G}(x))(x),\tag{AP}\label{eq:Apaxiomb}
\end{equation}
for any $\mathcal{F}\in\mathbb{F}X$ and any $\G(\cdot):X\to\mathbb{F}X$.

The category \textbf{Top} of topological spaces (with continuous maps)
is a reflective and coreflective (via the restrictions of $r$ and
$c$) subcategory of the category \textbf{Ap} of approach spaces \cite{Lowen89}.
There are several other equivalent descriptions of \textbf{Ap} and
\textbf{Prap} (see \cite{AP.book} and \cite{Lowen97} for details).

\subsection{Measures of compactness}

The \emph{adherence function of a filter }$\H$ in a convergence approach
space $(X,\lambda)$ is 
\begin{equation}
\adh_{\lambda}\mathcal{H}(\cdot)=\bigwedge\limits _{\mathcal{G}\#\mathcal{H}}\lambda(\G)(\cdot)=\bigwedge\limits _{\mathcal{U}\in\mathbb{U}(\mathcal{H})}\lambda(\mathcal{U})(\cdot).\label{eq:adh}
\end{equation}

We are now in a position to recall the main definitions and results
of \cite{mynardmeasureCAP}. 

Let $(X,\lambda)$ be a \textbf{Cap}-object, and let $\mathbb{D}$
be a class of filters. The \emph{measure of }$\mathbb{D}$\emph{-compactness
of a filter }$\mathcal{F}$ \emph{at }$\mathcal{A\subset}[0,\infty]^{X}$
is
\begin{equation}
c_{\mathbb{D}}^{\mathcal{A}}(\mathcal{F})=\bigvee\limits _{\mathbb{D}\ni\mathcal{D}\#\mathcal{F}}\bigvee\limits _{\phi\in\mathcal{A}}\bigwedge\limits _{x\in X}\left(\adh_{\lambda}\mathcal{D}+\phi\right)(x).\label{eq:Dcompmeasure}
\end{equation}

This definition is motivated by the special case where $\lambda=r(\lambda)$
and $\mathcal{A}\subset2^{X}$ via the identification of $A\subset X$
with the \emph{indicator function} $\theta_{A}$ \emph{of} $A$ taking
the value $0$ on $A$ and $\infty$ on $A^{c}.$ In this case, a
filter $\mathcal{F}$ is $\mathbb{D}$-compact at $\mathcal{A}$ (in
the sense of \cite{D.comp}) if and only if $c_{\mathbb{D}}^{\mathcal{A}}(\mathcal{F})=0.$
By a convenient abuse of notation, we will write $c_{\mathbb{D}}^{A}(\mathcal{F})$
for $c_{\mathbb{D}}^{\{\theta_{A}\}}(\mathcal{F})$ whenever $A\subset X.$
Notice that
\begin{equation}
c_{\mathbb{D}}^{A}(\mathcal{F})=\bigvee\limits _{\mathbb{D}\ni\mathcal{D}\#\mathcal{F}}\bigwedge\limits _{a\in A}\adh_{\lambda}\mathcal{D}(a).\label{eq:measureatset}
\end{equation}

In this paper, we are primarily concerned with measures of compactness
at a set, like in (\ref{eq:measureatset}).

When particularized to principal filters in approach spaces, the
measures of compactness \emph{\cite{lowen.KMC} }and\emph{ }relative
compactness (for $\mathbb{D}=\mathbb{F}$ the class of all filters),
relative countable compactness (for $\mathbb{D}=\mathbb{F}_{1}$ the
class of countably based filters), and relative sequential compactness\emph{
\cite{baekelandlowen}, }as well as the measure of Lindel\"of \emph{\cite{baekelandlowen.LS}
}(for $\mathbb{D}=\mathbb{F}_{\wedge1}$ the class of \emph{countably
deep }filters, that is, those closed under countable intersections)
are all instances, as shown in \cite[Examples 4-8]{mynardmeasureCAP}.

A subset $A$ of a $\mathbf{Cap}$-space $X$ is \emph{$\mathbb{D}$-compact}
if $c_{\mathbb{D}}^{A}A=0$ and \emph{relatively $\mathbb{D}$-compact
}if $c_{\mathbb{D}}^{X}A=0$. In particular, if $\mathbb{D}=\mathbb{F}$
is the class of all filters, we call $A$ \emph{compact} if $c_{\mathbb{F}}^{A}A=0$,
in contrast to the terminology of R. Lowen and his collaborators who
normally call such a set $0$-compact (e.g., \cite{lowen.Mrowka}),
and reserve the term compact for the smaller class of spaces whose
topological coreflection is compact in the topological sense.

\subsection{Endoreflectors of $\mathbf{Conv}$ and $\mathbf{Cap}$}

S. Dolecki presented in \cite{quest2} a unified treatment of several
important concrete endoreflectors and endocoreflectors of \textbf{Conv}.
In particular, given a class $\mathbb{J}$ of filters, he defined
the modifications $\Adh_{\mathbb{J}}\xi$ and $\Base_{\mathbb{J}}\xi$
of a convergence $\xi$ on $X$ as follows:
\[
\lim\nolimits _{\Adh_{\mathbb{J}}\xi}\mathcal{F}=\bigcap\limits _{\mathbb{J}\ni\mathcal{J}\#\mathcal{F}}\adh_{\xi}\mathcal{J},
\]
 where $\adh_{\xi}\J:=\bigcup_{\U\in\mathbb{U}(\J)}\lim_{\xi}\U$,
and 
\[
\lim\nolimits _{\Base_{\mathbb{J}}\xi}\mathcal{F}=\bigcup\limits _{\mathbb{J\ni}\mathcal{J\leq F}}\lim\nolimits _{\xi}\mathcal{J}.
\]
If the class $\mathbb{J}$ is independent of the convergence, stable
by finite infimum and stable by relation (%
\footnote{see \cite{quest2} for more general conditions.%
}), then $\Adh_{\mathbb{J}}$ is (the restriction to objects of) a
reflector and $\Base_{\mathbb{J}}$ is (the restriction to objects
of) a coreflector. In particular, when $\mathbb{J}$ is respectively
the class $\mathbb{F}$ of all filters, the class $\mathbb{F}_{1}$
of countably based filters and the class $\mathbb{F}_{0}$ of principal
filters, then $\Adh_{\mathbb{J}}$ is the reflector from \textbf{Conv
}onto the category of pseudotopological, paratopological and pretopological
spaces respectively; and $\Base_{\mathbb{J}}$ is the identity functor
of \textbf{Conv}, the coreflector from \textbf{Conv }onto first-countable
convergence spaces and the coreflector from \textbf{Conv }onto finitely
generated convergence spaces, respectively.

As observed in \cite{Mynard.survey}, the definitions of the reflectors
$\mathrm{Adh}_{\mathbb{J}}$ and of the coreflectors $\mathrm{Base}_{\mathbb{J}}$
extend from \textbf{Conv} to \textbf{Cap} via 
\begin{equation}
(\mathrm{Adh}_{\mathbb{J}}\lambda)(\mathcal{F})(x)=\underset{\mathbb{J}\ni\mathcal{H}\#\mathcal{F}}{\bigvee}\adh_{\lambda}\mathcal{H}(x),\label{eq:CAPAdhJ}
\end{equation}
and 
\begin{equation}
(\mathrm{Base}_{\mathbb{J}}\lambda)(\mathcal{F})(\cdot)=\underset{\mathbb{J}\ni\mathcal{G}\leq\mathcal{F}}{\bigwedge}\lambda(\mathcal{G})(\cdot).\label{eq:BaseJ}
\end{equation}

When $\mathbb{J}$ is respectively the class of all filters and of
principal filters, $\mathrm{Adh}_{\mathbb{J}}$ is respectively the
reflector on \textbf{Psap} and on \textbf{Prap}. Moreover, the category
\textbf{Parap} of para-approach spaces is introduced as the category
of fixed points for $\mathrm{Adh}_{\mathbb{J}}$ with the class $\mathbb{J}$
of countably based filters.Notice that (\ref{eq:CAPAdhJ}) gives an
explicit description of the reflection of a \textbf{Cap}-object on
\textbf{Psap}, \textbf{Parap} or \textbf{Prap}, but not on \textbf{Ap}. 

A convergence approach space $(X,\lambda)$ is called $\mathbb{J}$-\emph{based
}if $\lambda=\Base_{\mathbb{J}}\lambda$ (equivalently, $\lambda\geq\Base_{\mathbb{J}}\lambda$).

Measures of $\mathbb{D}$-compactness for filters generalize both
usual measure of compactness for sets and approach limits. It is this
very fact that allows to derive a variety of corollaries from any
result on the measure of $\mathbb{D}$-compactness of filters. With
our definitions, it is immediate that:
\begin{thm}
\label{th:AdhJarecompact}\cite[Theorem 9]{mynardmeasureCAP}
\[
\left(\Adh_{\mathbb{J}}\lambda\right)\left(\mathcal{F}\right)(x)=c_{\mathbb{J}}^{\{x\}}(\mathcal{F}).
\]

\end{thm}
The \textbf{Ap}-reflection of a convergence-approach space can also
be characterized in similar terms \cite[Theorem 10]{mynardmeasureCAP}.

\subsection{Calculus of relations}

Recall that $R\subseteq X\times Y$ can be seen as a multivalued map
$R:X\rightrightarrows Y$ with $y\in R(x)$ whenever $(x,y)\in R$.
We denote $R^{-}:Y\rightrightarrows X$ the inverse relation. If $R:X\rightrightarrows Y$
is a relation and $\F\in\filter X$ then 
\[
R[\F]:=\left\{ R(F):=\bigcup_{x\in F}R(x):\, F\in\F\right\} ^{\uparrow}
\]
is a (possibly degenerate) filter on $Y$. Note that if $\F\in\filter X$
and $\G\in\filter Y$, then 
\begin{equation}
(\F\times\G)\#R\Longleftrightarrow R[\F]\#\G\Longleftrightarrow\F\#R^{-}[\G].\label{eq:productmeshR}
\end{equation}
More generally, if $\J\in\filter(X\times Y)$ then each element of
$\J$ can be seen as a relation, and we can define
\[
\J[\F]:=\left\{ J(F):\, F\in\F,\, J\in\J\right\} ^{\uparrow}
\]
which is a (possibly degenerate) filter on $Y$. $\J^{-}[\G]$ is
defined similarly and is a (possibly degenerate) filter on $X$. With
these notations, (\ref{eq:productmeshR}) immediately extends to 
\begin{equation}
(\F\times\G)\#\J\Longleftrightarrow\J[\F]\#\G\Longleftrightarrow\F\#\J^{-}[\G].\label{eq:productmeshgeneral}
\end{equation}
As a general convention, all classes of filters contain the degenerate
filter of each set. 

Let $\mathbb{D}$ and $\mathbb{J}$ be two classes of filters. Then
$\mathbb{D}$ is a \emph{$\mathbb{J}$-composable} class of filters
if for every pair of sets $X$ and $Y$, when $D\in\mathbb{D}X$ and
$\mathcal{J}\in\mathbb{J}(X\times Y)$, $\mathcal{J}[\D]\in\mathbb{D}Y$.
In particular, if $\mathbb{D}$ is $\mathbb{D}$-composable, we say
that $\mathbb{D}$ is \emph{composable}. For instance the classes
$\mathbb{F}_{0}$ of principal filters, $\mathbb{F}_{1}$ of countably
based filters, and $\mathbb{F}_{\wedge1}$ of countably deep filters
are all composable classes containing $\mathbb{F}_{0}$, so that they
are in particular $\mathbb{F}_{0}$-composable. In contrast, the class
$\mathbb{E}$ of sequential filters is not $\mathbb{F}_{0}$-composable.

\section{Compact relations in $\mathbf{Cap}$}

As we set out to extend some of the results of \cite{myn.applofcompact}
from $\mathbf{Conv}$ to $\mathbf{Cap}$, a necessary first step is
to extend to $\mathbf{Cap}$ the characterizations of various types
of quotient maps and of perfect-like maps in terms of preservation
of compactness established in \cite{myn.relations}. This is the purpose
of this section.

A relation $R:X\rightrightarrows Y$ is $\mathbb{D}$-\emph{compact}
if for every $\F\in\filter X$ and $A\subset X,$ 
\[
c_{\mathbb{D}}^{R[A]}(R[\F])\leq c_{\mathbb{D}}^{A}(\F).
\]

\begin{lem}
\label{lem:decreaselevelofcompactness} Let $R:X\rightrightarrows Y$
be a $\mathbb{D}$-compact relation and $\mathbb{J}\subset\mathbb{D}$
be classes of filters with $\mathbb{J}$ an $\mathbb{F}_{0}$-composable
class. Then $R$ is also $\mathbb{J}$-compact. \end{lem}
\begin{proof}
Let $\F\in\mathbb{F}X$ and $\G\in\mathbb{J},$ with $\G\#R[\F]$.
By (\ref{eq:productmeshR}), $R^{-}[\G]\#\F$ and $R^{-}[\G]\in\mathbb{J}$
since $\mathbb{J}$ is an $\mathbb{F}_{0}$-composable class. Therefore
$\bigwedge_{x\in A}\adh\, R^{-}[\G](x)\leq c_{\mathbb{J}}^{A}(\F),$
so that for every $\epsilon>0,$ there is an $x_{\epsilon}\in A$
and a $\U_{\epsilon}\in\mathbb{U}(R^{-}[\G])$ such that 
\[
c_{\mathbb{J}}^{A}(\F)+\epsilon\geq\lambda_{X}(\U_{\epsilon})(x_{\epsilon}).
\]
By the $\mathbb{D}$-compactness of the relation $R,$ we see that
\begin{equation}
\lambda_{X}(\U_{\epsilon})(x_{\epsilon})\geq c_{\mathbb{D}}^{A}(\U_{\epsilon})\geq c_{\mathbb{D}}^{R[A]}(R[\U_{\epsilon}]).\label{eq:aux1}
\end{equation}
Since $\mathbb{J}\subset\mathbb{D}$, $c_{\mathbb{D}}^{B}\H\geq c_{\mathbb{J}}^{B}\H$
for any filter $\H$ and set $B$. Moreover, $R[\U_{\epsilon}]\#\G,$
so that, in view of (\ref{eq:aux1}), 
\[
\lambda_{X}(\U_{\epsilon})(x_{\epsilon})\geq c_{\mathbb{J}}^{R[A]}(R[\U_{\epsilon}])\geq\bigwedge_{y\in R[A]}\adh_{Y}\G(y).
\]
Comparing the extremes of these inequalities, we have that 
\[
\bigwedge_{y\in R[A]}\adh_{Y}\G(y)\leq c_{\mathbb{J}}^{A}\F+\epsilon,
\]
for every $\epsilon,$ so that $c_{\mathbb{J}}^{R[A]}R[\F]\leq c_{\mathbb{J}}^{A}\F$. \end{proof}
\begin{lem}
\label{thm:lem2} Let $\mathbb{D}$ be an $\filter_{0}$-composable
class of filters. Then $R:X\rightrightarrows Y$ is a $\mathbb{D}$-compact
relation if and only if
\begin{equation}
\lambda_{X}(\F)(x)\geq c_{\mathbb{D}}^{R(x)}R[\F]\label{eq:convtocomp}
\end{equation}
 for every $\F\in\filter X$ and every $x\in X$. \end{lem}
\begin{proof}
Assume that $R$ is a $\mathbb{D}$-compact relation. For every $x\in X$
and $\F\in\filter X,$

\begin{eqnarray*}
c_{\mathbb{D}}^{R(x)}R[\F] & \leq & c_{\mathbb{D}}^{\{x\}}(\F)\\
 & = & \bigvee_{\mathbb{D}\ni\D\#\F}\bigwedge_{\G\#\D}\lambda_{X}(\G)(x)\\
 & \leq & \lambda_{X}(\F)(x),
\end{eqnarray*}

because $\F\#\D$.

Conversely, assume (\ref{eq:convtocomp}) for every $\F\in\mathbb{F}X$
and every $x\in X$, and given a filter $\F\in\mathbb{F}X$, consider
a $\mathbb{D}$-filter $\D$ that meshes with $R[\F]$. By (\ref{eq:productmeshR}),
$R^{-}[\D]\#\F$, and $R^{-}[\D]\in\mathbb{D}$ because $\mathbb{D}$
is $\filter_{0}$-composable, so that for every $A\subset X,$ 
\[
\bigwedge_{x\in A}\adh_{X}R^{-}[\D](x)\leq c_{\mathbb{D}}^{A}(\F).
\]
Thus, for any $\epsilon>0,$ there is an $x_{\epsilon}\in A$ and
$\U_{\epsilon}\in\mathbb{U}(R^{-}[\D])$ such that 
\[
\lambda(\U_{\epsilon})(x_{\epsilon})\leq c_{\mathbb{D}}^{A}(\F)+\epsilon.
\]
By (\ref{eq:convtocomp}) applied to $\U_{\epsilon}$ and $x_{\epsilon}$,
\[
c_{\mathbb{D}}^{R(x_{\epsilon})}R[\U_{\epsilon}]\leq c_{\mathbb{D}}^{A}(\F)+\epsilon.
\]
Moreover, $R[\U_{\epsilon}]\#\D$ because $\U_{\epsilon}\#R^{-}[\D],$
so that 
\[
\bigwedge_{y\in R(A)}\adh_{Y}\D(y)\leq\bigwedge_{y\in R(x_{\epsilon})}\adh_{Y}\D(y)\leq c_{\mathbb{D}}^{R(x_{\epsilon})}R[\U_{\epsilon}].
\]
We conclude that for every $\epsilon>0$, 
\[
\bigvee_{\mathbb{D}\ni\D\#R[\F]}\bigwedge_{y\in R(A)}\adh_{Y}\D(y)=c_{\mathbb{D}}^{R(A)}R[\F]\leq c_{\mathbb{D}}^{A}(\F)+\epsilon,
\]
which yields the desired property that $c_{\mathbb{D}}^{R(A)}R[\F]\leq c_{\mathbb{D}}^{A}(\F)$. 
\end{proof}

\begin{cor}
\label{cor:continuous} Let $\mathbb{D}$ be an $\mathbb{F}_{0}$-composable
class of filters and let $f:(X,\lambda_{X})\rightarrow(Y,\lambda_{Y})$
with $Y=\Adh_{\mathbb{D}}Y$. The following are equivalent:
\begin{enumerate}
\item $f$ is a contraction;
\item $f$ is a compact relation;
\item $f$ is a $\mathbb{D}$-compact relation. 
\end{enumerate}
\end{cor}
\begin{proof}
$(1\Longrightarrow2).$ If $f$ is a contraction then 
\[
\lambda_{X}(\F)(x)\geq\lambda_{Y}(f[\F])(f(x)\ge c_{\mathbb{F}}^{\{f(x)\}}f[\F]
\]
and Lemma \ref{thm:lem2} applies to the effect that $f$ is a compact
relation. $(2\Longrightarrow3)$ is obvious, and $(3\Longrightarrow1)$
follows from Theorem \ref{th:AdhJarecompact} and $Y=\Adh_{\mathbb{D}}Y$. 
\end{proof}
In particular, $\mathbb{F}_{0}$-compact (equivalently compact) maps
between pre-approach spaces are exactly the contractive ones.

Since $c_{\mathbb{D}}^{\{x\}}\F\leq\lambda(\F)(x),$ Lemma \ref{thm:lem2}
immediately gives:
\begin{cor}
\label{thm:Dcomcor} Let $\mathbb{D}$ be an $\filter_{0}$-composable
class, then $R:X\rightrightarrows Y$ is a $\mathbb{D}$-compact relation
if and only if for every $\F\in\filter X$, and $x\in X$, $c_{\mathbb{D}}^{\{x\}}(\F)\geq c_{\mathbb{D}}^{R(x)}R[\F].$ 
\end{cor}
If $\mathbb{D}$ is a class of filters that contains $\mathbb{F}_{0}$,
then, in view of Lemma \ref{lem:decreaselevelofcompactness}, a $\mathbb{D}$-compact
relation $R$ is also $\mathbb{F}_{0}$-compact, and for each $x\in X$,
\[
c_{\mathbb{D}}^{R(x)}R(x)\leq c_{\mathbb{D}}^{\{x\}}\{x\}=0,
\]
so that $R(x)$ is a $\mathbb{D}$-compact subset of $Y$. When $Y$
is an approach space, the converse is true:
\begin{thm}
\label{thm:Dcomprel} Let $X$ be a convergence approach space, let
$Y$ be an\emph{ approach} space, let $\mathbb{D}$ be an $\filter_{0}$-composable
class of filters, and let $R:X\rightrightarrows Y$ be an $\mathbb{F}_{0}$-compact
relation. If for every $x\in X,\, R(x)$ is a $\mathbb{D}$-compact
subset of $Y$, then $R$ is a $\mathbb{D}$-compact relation. \end{thm}
\begin{proof}
In view of \ref{thm:Dcomcor}, it suffices to show that $c_{\mathbb{D}}^{\{x\}}(\F)\geq c_{\mathbb{D}}^{R(x)}R[\F]$
for every $x\in X$ and $\F\in\filter X,$ equivalently, that given
$x\in X$ and $\F\in\mathbb{F}X$, for every $\D\in\mathbb{D}$ with
$\D\#R[\F]$, 
\begin{equation}
\bigwedge_{y\in R(x)}\adh_{Y}\D(y)\leq c_{\mathbb{D}}^{\{x\}}(\F).\label{eq:aux2}
\end{equation}
By (\ref{eq:productmeshR}), $R^{-}[\D]\#\F$, and $R^{-}[\D]\in\mathbb{D}$
because $\mathbb{D}$ is an $\mathbb{F}_{0}$-composable class. Thus
$\adh_{X}R^{-}[\D](x)\leq c_{\mathbb{D}}^{\{x\}}(\F).$

For every $\epsilon>0,$ there is a $\U_{\epsilon}\#R^{-}[\D]$ such
that $\lambda_{X}(\U_{\epsilon})(x)\leq c_{\mathbb{D}}^{\{x\}}(\F)+\epsilon.$
Since $R$ is an $\filter_{0}$-compact relation, 
\[
c_{\mathbb{F}_{0}}^{R(x)}R[\U_{\epsilon}]\leq c_{\mathbb{F}_{0}}^{\{x\}}(\U_{\epsilon})\leq\lambda(\U_{\epsilon})(x)\leq c_{\mathbb{D}}^{\{x\}}(\F)+\epsilon.
\]

Since $\D\#R[\U_{\epsilon}]$, then for every $D\in\D,$ we have that
$\bigwedge_{y\in R(x)}\adh_{Y}D(y)\leq c_{\mathbb{F}_{0}}^{R(x)}R[\U_{\epsilon}]$.
Setting $\alpha=c_{\mathbb{D}}^{\{x\}}(\F)+\epsilon$ and $\D^{(\alpha)}=\{D^{(\alpha)}:\, D\in\D\}\in\mathbb{D}$,
we have that $\bigwedge_{y\in R(x)}\adh_{Y}\D^{(\alpha)}(y)=0$ because
$\D^{(\alpha)}\#R(x)$ and $R(x)$ is a $\mathbb{D}$-compact subset
of $Y$. Since $Y$ is an approach space, 
\[
\adh_{Y}\D(y)\leq\adh_{Y}\D^{(\alpha)}(y)+\alpha
\]
so that, given that $\bigwedge_{y\in R(x)}\adh_{Y}\D^{(\alpha)}(y)=0$, 

\begin{eqnarray*}
\alpha & = & \left(\bigwedge_{y\in R(x)}\adh_{Y}\D^{(\alpha)}(y)\right)+\alpha\\
 & = & \bigwedge_{y\in R(x)}\left(\adh_{Y}\D^{(\alpha)}(y)+\alpha\right)\\
c_{\mathbb{D}}^{\{x\}}(\F)+\epsilon=\alpha & \geq & \bigwedge_{y\in R(x)}\adh_{Y}\D(y).
\end{eqnarray*}

Since this inequality is true for every $\epsilon>0$ we obtain (\ref{eq:aux2})
and the conclusion follows.\end{proof}
\begin{cor}
\label{cor:Dcompactrelationschar}Let $X$ be a convergence approach
space, let $Y$ be an\emph{ approach} space, let $\mathbb{D}$ be
an $\filter_{0}$-composable class of filters. Then a relation $R:X\rightrightarrows Y$
is $\mathbb{D}$-compact if and only if it is $\mathbb{F}_{0}$-compact
and for every $x\in X,\, R(x)$ is a $\mathbb{D}$-compact subset
of $Y$.\end{cor}
\begin{proof}
If $\mathbb{D}$ is $\mathbb{F}_{0}$-composable, then in particular
$\mathbb{F}_{0}\subseteq\mathbb{D}$. Indeed, Let $A\in\mathbb{F}_{0}(X)$
and let $\D$ be a non-degenerate $\mathbb{D}$-filters on $X$. Let
$R:=X\times A$. Then $R[\D]=A$. Thus, Lemma \ref{lem:decreaselevelofcompactness}
applies to the effect that if $R$ is a $\mathbb{D}$-compact relation,
it is also $\mathbb{F}_{0}$-compact, and as observed before each
$R(x)$ is $\mathbb{D}$-compact. The converse is Theorem \ref{thm:Dcomprel}.
\end{proof}

\subsection{Closed and perfect maps}

Lowen et al. introduced in \cite{lowen.Mrowka} a notion of closed
maps in $\mathbf{Ap}$ and checked that this class of morphisms satisfies
the conditions to be a categorically well-behaved class of closed
morphisms in the sense of \cite{functionaltopology}. Namely, a map
$f:(X,\lambda_{X})\to(Y,\lambda_{Y})$ between two approach spaces
is \emph{closed-expansive}, which we will abridge as \emph{closed},
if for every $y\in Y$ and $A\subset X$ 
\begin{equation}
\bigwedge_{x\in f^{-}(y)}\adh_{X}\, A(x)\leq\adh_{Y}\, f(A)(y).\label{eq:closedmap}
\end{equation}
 We extend this definition to convergence approach spaces.
\begin{prop}
\label{prop:closedF0compact} A map $f:(X,\lambda_{X})\to(Y,\lambda_{Y})$
between convergence approach spaces is closed (in the sense of (\ref{eq:closedmap}))
if and only if $f^{-}:Y\rightrightarrows X$ is an $\mathbb{F}_{0}$-compact
relation.\end{prop}
\begin{proof}
Assume that $f:X\to Y$ is closed. According to Lemma \ref{thm:lem2},
we need to show that for every $\G\in\mathbb{F}Y$ and $y\in Y$,
$c_{\mathbb{F}_{0}}^{f^{-}(y)}f^{-}[\G]\leq\lambda_{Y}(\G)(y)$. If
$A\#f^{-}[\G]$, then $f(A)\#\G$ so that $\adh_{Y}f(A)(y)\leq\lambda_{Y}(\G)(y)$
and, in view of (\ref{eq:closedmap}), 
\[
\bigwedge_{x\in f^{-}(y)}\adh_{X}\, A(x)\leq\lambda_{Y}(\G)(y).
\]
Since this is true for every $A\#f^{-}[\G]$, $c_{\mathbb{F}_{0}}^{f^{-}(y)}f^{-}[\G]\leq\lambda_{Y}(\G)(y).$ 

Conversely, assume that $f^{-}:Y\rightrightarrows X$ is an $\mathbb{F}_{0}$-compact
relation. To show (\ref{eq:closedmap}), note that if $\G\#f(A)$
then $f^{-}[\G]\#A$ so that, by $\mathbb{F}_{0}$-compact of $f^{-}$
and Lemma \ref{thm:lem2}, $\bigwedge_{x\in f^{-}(y)}\adh_{X}\, A(x)\leq\lambda_{Y}(\G)(y),$
for each $\G\#f(A)$, so that 
\[
\bigwedge_{x\in f^{-}(y)}\adh_{X}\, A(x)\leq\bigwedge_{\G\#f(A)}\lambda_{Y}(\G)(y)=\adh_{Y}f(A)(y).
\]

\end{proof}
Let us call a map $f:X\to Y$ between convergence approach spaces
\emph{$\mathbb{D}$-perfect} if $f^{-}:Y\rightrightarrows X$ is a
$\mathbb{D}$-compact relation. In view of Proposition \ref{prop:closedF0compact},
$\mathbb{F}_{0}$-perfect means closed. Moreover, Corollary \ref{cor:Dcompactrelationschar}
applies to $f^{-}$ to the effect that:
\begin{thm}
\label{thm:Dperfect} Let $\mathbb{D}$ be an $\mathbb{F}_{0}$-composable
class of filters, $X$ be an \emph{approach} space and $Y$ be a convergence
approach space. A map $f:X\to Y$ is $\mathbb{D}$-perfect if and
only if $f$ is closed and for every $y\in Y,\, f^{-}y$ is $\mathbb{D}$-compact. 
\end{thm}

\subsection{Quotient maps}

S. Dolecki observed in \cite{quest2} that the notions of quotient,
hereditarily quotient, countably biquotient, biquotient, almost open
maps (in the sense of \cite{quest}) can be extended from the category
$\mathbf{Top}$ of topological spaces to the category $\mathbf{Conv}$
of convergence spaces by noting that a map between topological spaces
is hereditarily quotient, countably biquotient, biquotient, almost
open respectively, if it is quotient when regarded in the category
of pretopological spaces, paratopological spaces, pseudotopological
spaces, and convergence spaces respectively. To fully extend the notions
to $\mathbf{Conv}$, he further observed that if $f:X\to Y$ is onto
between topological spaces, seen as convergence spaces, the map is
quotient in a reflective subcategory if the convergence of $Y$ is
finer than the reflection of the final convergence for $f$ and $X$. 

We can proceed exactly the same way in $\mathbf{Cap}$: Given a map
$f:(X,\lambda_{X})\to Y$, there is the finest limit function $\lambda_{fX}$
on $Y$ making $f$ a contraction, that is, the quotient structure
in $\mathbf{Cap}$. Given an $\mathbb{F}_{0}$-composable class $\mathbb{D}$,
$\Adh_{\mathbb{D}}$ (given by (\ref{eq:CAPAdhJ})) defines a reflector,
and the inequality 
\begin{equation}
\lambda_{Y}\geq\Adh_{\mathbb{D}}\lambda_{fX}\label{eq:Dquotient}
\end{equation}
characterizes the fact that a surjective map $f:(X,\lambda_{X})\to(Y,\lambda_{Y})$
is quotient in the full reflective category of $\mathbf{Cap}$ of
objects fixed by $\Adh_{\mathbb{D}}$. When $\mathbb{D=F}_{0}$, (\ref{eq:Dquotient})
defines hereditarily quotient maps between $\mathbf{Cap}$ spaces.
Of course, a map between two topological spaces is hereditarily quotient
if and only if it is hereditarily quotient when domain and codomain
are seen as convergence approach spaces. Similarly, (\ref{eq:Dquotient})
for $\mathbb{D}=\mathbb{F}_{1}$ defines countably biquotient maps,
and (\ref{eq:Dquotient}) for $\mathbb{D}=\mathbb{F}$ defines biquotient
maps. Naturally, we call a surjective map $f:(X,\lambda_{X})\to(Y,\lambda_{Y})$
$\mathbb{D}$-\emph{quotient }if (\ref{eq:Dquotient}) holds. 

$\mathbb{D}$-quotient maps are also instances of $\mathbb{D}$-compact
relations. To see that, we need to use both the final $\mathbf{Cap}$
structure $\lambda_{fX}$ but also the initial $\mathbf{Cap}$ structure
$\lambda_{f^{-}Y}$, that is, the coarsest $\mathbf{Cap}$ structure
on $X$ making the map $f:X\to(Y,\lambda_{Y})$ a contraction. Initial
and final structures in $\mathbf{Cap}$ are described in \cite[Proposition 2.3]{lowe88}
to the effect that 
\begin{equation}
\lambda_{f^{-}Y}(\F)(x)=\lambda_{Y}(f[\F])(f(x))\label{eq:initial}
\end{equation}
and 
\begin{equation}
\lambda_{fX}(\G)(y)=\begin{cases}
0 & \mbox{if }\G=\{y\}^{\uparrow}\\
\bigwedge_{x\in f^{-}y}\bigwedge_{\F\in\mathbb{F}X:\, f[\F]\leq\G}\lambda_{X}(\F)(x) & \mbox{otherwise.}
\end{cases}\label{eq:final}
\end{equation}

\begin{lem}
\label{lem:adhfX}If $f:(X,\lambda_{X})\to(Y,\lambda_{Y})$ is onto
and $\D\in\mathbb{F}Y$
\[
\adh_{fX}\D(y)=\bigwedge_{x\in f^{-}y}\adh_{X}f^{-}[\D](x).
\]
\end{lem}
\begin{proof}
Since $f$ is onto, for each $\U\in\mathbb{U}(\D)$ there is $\W\in\mathbb{U}(f^{-}[\D])$
with $f[\W]=\U$ so that 
\[
\lambda_{fX}\U(y)=\bigwedge_{x\in f^{-}y}\lambda_{X}\W(x)\geq\bigwedge_{x\in f^{-}y}\adh_{X}f^{-}[\D](x).
\]

On the other hand, if $\W\in\mathbb{U}(f^{-}[\D])$ then $f[\W]\in\mathbb{U}(\D)$
and $f:(X,\lambda_{X})\to(Y,\lambda_{fX})$ is a contraction, so that,
for every $x\in f^{-}y$,
\[
\lambda_{X}\W(x)\geq\lambda_{fX}f[\W](y)\geq\adh_{fX}\D(y)
\]
and we conclude that 
\[
\bigwedge_{x\in f^{-}y}\adh_{X}f^{-}[\D](x)\geq\adh_{fX}\D(y).
\]
\end{proof}
\begin{thm}
\label{th:Dquotient} Let $\mathbb{D}$ be an $\mathbb{F}_{0}$-composable
class of filters. Let $f:(X,\lambda_{X})\rightarrow(Y,\lambda_{Y})$
be onto. Then $f$ is $\mathbb{D}$-quotient if and only if $f:(X,\lambda_{f^{-}Y})\rightarrow(Y,\lambda_{fX})$
is a $\mathbb{D}$-compact relation. \end{thm}
\begin{proof}
Assume $f$ is $\mathbb{D}$-quotient. Then given $\F\in\mathbb{F}X$
and $x\in X$, 
\[
\lambda_{f^{-}Y}(\F)(x)=\lambda_{Y}(f[\F])(f(x))\geq\Adh_{\mathbb{D}}\lambda_{fX}(f[\F])(f(x))=c_{\mathbb{D}}^{\{f(x)\}}f[\F]
\]
where the measure of $\mathbb{D}$-compactness is in $(Y,\lambda_{fX})$.
In view of Lemma \ref{thm:lem2}, $f:(X,\lambda_{f^{-}Y})\rightarrow(Y,\lambda_{fX})$
is a $\mathbb{D}$-compact relation. 

Conversely, assume that $f:(X,\lambda_{f^{-}Y})\rightarrow(Y,\lambda_{fX})$
is a $\mathbb{D}$-compact relation. Let $\G\in\mathbb{F}Y$ and $y\in Y$.
Since $f$ is onto, $\G=f[f^{-}[\G]]$, and there is $x\in f^{-}y$,
so that 
\[
\lambda_{Y}(\G)(y)=\lambda_{f^{-}Y}(f^{-}[\G])(x)\geq c_{\mathbb{D}}^{\{f(x)\}}\G,
\]
where the right hand side is measured in $(Y,\lambda_{fX})$. In view
of Theorem \ref{th:AdhJarecompact}, $\lambda_{Y}\geq\Adh_{\mathbb{D}}\lambda_{fX}$.
\end{proof}
Note that Corollary \ref{cor:Dcompactrelationschar} does not apply
to $\mathbb{D}$-quotient maps in general, for even if $X$ is an
approach space, $(Y,\lambda_{fX})$ generally fails to be.

The table below gathers the terminology we use for various instances
of $\mathbb{F}_{0}$-composable classes of filters:\bigskip{}

\begin{table}[h!]
\centering{}%
\begin{tabular}{|c|c|c|c|c|}
\hline 
Class $\mathbb{D}$ of filters & $\Adh_{\mathbb{D}}$-fixed spaces  & $\mathbb{D}$-compact filter & \emph{$\mathbb{D}$-}quotient map & $\mathbb{D}$-perfect map \tabularnewline
\hline 
$\mathbb{F}$  & \textbf{Psap} & Compact  & biquotient & Perfect\tabularnewline
(All filters)  & (pseudo-approach) &  &  & \tabularnewline
\hline 
$\mathbb{F}_{\wedge1}$ & \textbf{Hypoap} & Lindel\"of & weakly & inversely \tabularnewline
(countably deep) & (hypo-approach) &  & biquotient & Lindel\"of\tabularnewline
\hline 
$\mathbb{F}_{1}$  & \textbf{Parap} & Countably  & countably & Countably \tabularnewline
(Countably Based)  & (para-approach) & compact  & biquotient & perfect\tabularnewline
\hline 
$\mathbb{F}_{0}$  & \textbf{Prap} & Finitely  & hereditarily & Closed\tabularnewline
(Principal)  & (pre-approach) & compact & quotient & \tabularnewline
\hline 
\end{tabular}
\end{table}

\begin{prop}
Let $\mathbb{D}$ be an $\mathbb{F}_{0}$-composable class of filters.
A $\mathbb{D}$-perfect surjective map is $\mathbb{D}$-quotient.
In particular, in $\mathbf{Cap}$, surjective perfect maps are biquotient
and surjective closed maps are hereditarily quotient, hence quotient
in $\mathbf{Ap}$.\end{prop}
\begin{proof}
Let $\F\in\mathbb{F}X$ and $x\in X$, and let $y:=f(x)$, $\G:=f[\F]$.
Then, in view of (\ref{eq:initial}), 
\[
\lambda_{f^{-}Y}\F(x)=\lambda_{Y}(\G)(y)\geq c_{\mathbb{D}}^{f^{-}y}f^{-}[\G]
\]
because $f^{-}:Y\rightrightarrows X$ is $\mathbb{D}$-compact. Moreover,
\[
c_{\mathbb{D}}^{f^{-}y}f^{-}[\G]=\bigvee_{\mathbb{D}\ni\D\#f^{-}[\G]}\bigwedge_{x\in f^{-}y}\adh_{X}\D(x)
\]
so that, if $\D\in\mathbb{D}(Y)$ with $\D\#\G$ then $f^{-}[\D]\#f^{-}[\G]$
and $f^{-}[\D]\in\mathbb{D}(X)$ by $\mathbb{F}_{0}$-composability,
so that
\[
\bigwedge_{x\in f^{-}y}\adh_{X}f^{-}[\D](x)\leq c_{\mathbb{D}}^{f^{-}y}f^{-}[\G]\leq\lambda_{Y}(\G)(y).
\]
 By Lemma \ref{lem:adhfX}, $\adh_{fX}\D(y)=\bigwedge_{x\in f^{-}y}\adh_{X}f^{-}[\D](x)$,
and we conclude that $\Adh_{\mathbb{D}}\lambda_{fX}\leq\lambda_{Y}$. 
\end{proof}

\section{Product of measures of compactness}

The main result to be applied in the next section is the following
extension from $\mathbf{Conv}$ to $\mathbf{Cap}$ of \cite[Theorem 1]{myn.applofcompact}:
\begin{thm}
\label{thm:mainproduct} Let $(X,\lambda_{X})$ be a convergence approach
space, $A\subset X$, $\alpha\in[0,\infty),$ and let $\F\in\mathbb{F}X.$
Let $\mathbb{D}$ be a composable class of filters that contains principal
filters. The following are equivalent:
\begin{enumerate}
\item $c_{\mathbb{D}}^{A}(\F)\leq\alpha$; 
\item For every convergence approach space $(Y,\lambda_{Y}),$ every $B\subset Y$,
and every $\G\in\mathbb{D}(Y)$, 
\[
c_{\mathbb{D}}^{A\times B}(\F\times\G)\leq\alpha\vee c_{\mathbb{F}}^{B}(\G);
\]

\item For every $\mathbb{D}$-based atomic topological approach space $Y$,
with non-isolated point $\infty$ and neighborhood filter $\N(\infty)$,
\[
c_{\mathbb{F}_{0}}^{A\times\{\infty\}}(\F\times\N(\infty))\leq\alpha.
\]

\end{enumerate}
\end{thm}
Note that the case where $\alpha=\infty$ is trivially true.
\begin{proof}
(1) $\then$ (2) Assume that $c_{\mathbb{D}}^{A}(\F)\leq\alpha.$
Let $\D\in\mathbb{D}(X\times Y$) with $\D\#(\F\times\G)$. By (\ref{eq:productmeshgeneral}),
$\D^{-}[\G]\#\F$ and moreover, $\D^{-}[\G]\in\mathbb{D}X$ because
$\G\in\mathbb{D}Y$ and $\mathbb{D}$ is composable. Since $c_{\mathbb{D}}^{A}(\F)\leq\alpha,$
for every $\epsilon>0$ there is a $\U_{\epsilon}\#\D^{-}[\G]$ and
$u_{\epsilon}\in A$ such that $\lambda_{X}(\U_{\epsilon})(u_{\epsilon})\leq\alpha+\epsilon.$
By (\ref{eq:productmeshgeneral}), , so that there is $\W_{\epsilon}\in\mathbb{U}(\D[\U_{\epsilon}])$
and $b_{\epsilon}\in B$ such that $\lambda_{Y}(\W_{\epsilon})(b_{\epsilon})\leq c_{\mathbb{F}}^{B}(\G)+\epsilon.$
Moreover, $\W_{\epsilon}\#\D[\U_{\epsilon}]$ so that $\D\#(\W_{\epsilon}\times\U_{\epsilon}),$
and thus 
\[
\bigwedge_{(a,b)\in A\times B}\adh_{A\times B}\,\D(a,b)\leq(\alpha+\epsilon)\vee(c_{\filter}^{B}(\G)+\epsilon).
\]
Since this is true for every $\epsilon>0,$ we obtain the result.

(2) $\then$ (3) is obvious because $\mathbb{F}_{0}\subseteq\mathbb{D}.$

(3) $\then$ (1) Assume that $c_{\mathbb{D}}^{A}(\F)>\alpha$, and
let $\beta$ be such that $c_{\mathbb{D}}^{A}(\F)>\beta>\alpha.$
Then there is a $\D\in\mathbb{D}$ with $\D\#\F$ such that 
\[
\bigwedge_{a\in A}\adh_{X}\D(a)\geq\beta.
\]

We construct a $\mathbb{D}$-based topological approach space with
underlying set $Y=X\cup\{\infty\}$ where $\infty\notin X$. Set every
point of $X$ to be isolated in $Y$, that is, for every $x\in X$,
$\lambda_{Y}(\F)(x)=\infty$ for every $\F\neq\{x\}^{\uparrow}$ and
$\lambda_{Y}(\{x\}^{\uparrow})(x)=0$. Let $\N_{Y}(\infty):=\D\wedge\{x\}^{\uparrow}\in\mathbb{D},$
that is, $\lambda_{Y}(\F)(\infty)=0$ if $\F\geq\N_{Y}(\infty)$ and
$\lambda_{Y}(\F)=\infty$ otherwise.

Let 
\[
\Delta:=\{(x,x):x\in X\}^{\uparrow}\in\mathbb{F}_{0}(X\times Y).
\]
 Note that $\Delta\#(\F\times\N_{Y}(\infty))$ because $\D\#\F$ in
$X$. We claim that 
\begin{equation}
\bigwedge_{a\in A}\adh_{X\times Y}\Delta(a,\infty)>\alpha\vee\lambda_{Y}(\N_{Y}(\infty))(\infty)=\alpha,\label{eq:claim}
\end{equation}
which yields $c_{\mathbb{F}_{0}}^{A\times\{\infty\}}(\F\times\N_{Y}(\infty))>\alpha.$ 

To verify (\ref{eq:claim}), note that for every $a\in A$,

\begin{eqnarray*}
\adh_{X\times Y}\Delta(a,\infty) & = & \bigwedge_{\Delta\leq\H}\lambda_{X\times Y}(\H)(a,\infty)\\
 & = & \bigwedge_{\Delta\leq\H}\lambda_{X}(p_{X}[\H])(a)\vee\lambda(p_{Y}[\H])(\infty).
\end{eqnarray*}
If $p_{Y}[\H]\neg\#\D$, then $\lambda_{Y}(p_{Y}[\H])(\infty)=\infty$
so $\adh_{X\times Y}\Delta(a,\infty)>\beta.$ Otherwise, $\D\#p_{Y}[\H]$
and $\Delta\leq\H$ so that $\D\#p_{X}[\H]$, and thus, $\lambda_{X}(p_{X}[\H])(a)\geq\beta$.
Either way, $\adh_{X\times Y}\Delta(a,\infty)\geq\beta$, proving
our claim.
\end{proof}
In order to apply this result to product of maps, we need the following
extension from $\mathbf{Conv}$ to $\mathbf{Cap}$ of \cite[Corollary 12]{myn.applofcompact}:
\begin{thm}
\label{thm:mainformaps} Let $\mathbb{D}$ be a composable class of
filters containing principal filters, and let $X$ and $Y$ be two
convergence approach spaces. The following are equivalent: 
\begin{enumerate}
\item $R:X\rightrightarrows Y$ is a $\mathbb{D}$-compact relation. 
\item For every $\mathbb{D}$-based convergence \emph{approach} space $Z$,
$R\times Id_{Z}:X\times Z\rightrightarrows Y\times Z$ is a $\mathbb{D}$-compact
relation. 
\item For every atomic topological $\mathbb{D}$-based approach space $Z,$
$R\times Id_{Z}:X\times Z\rightrightarrows Y\times Z$ is an $\mathbb{F}_{0}$-compact
relation. 
\end{enumerate}
\end{thm}
\begin{proof}
$(1\Rightarrow2)$ Let $R:X\rightrightarrows Y$ be a $\mathbb{D}$-compact
relation. We  show that
\[
\lambda_{X}(\F)(x)\vee\lambda_{Z}(\G)(z)\geq c_{\mathbb{D}}^{R(x)\times\{z\}}(R[\F]\times\G)
\]
 for every $\F\in\filter X$, $\G\in\filter Z$, $x\in X$, and $z\in Z$.
To this end, note that for every $\D\in\mathbb{D}$ such that $\D\leq\G$, 

\begin{eqnarray*}
c_{\mathbb{D}}^{R(x)\times\{z\}}(R[\F]\times\G) & \leq & c_{\mathbb{D}}^{R(x)\times\{z\}}(R[\F]\times\D)\\
 & \leq & c_{\mathbb{D}}^{R(x)}(R[\F])\vee c_{\mathbb{F}}^{\{z\}}(\D)\mbox{ by Theorem \ref{thm:mainproduct}}\\
 & \leq & \lambda_{X}(\F)(x)\vee\lambda_{Z}(\D)(z)\mbox{ so that}\\
c_{\mathbb{D}}^{R(x)\times\{z\}}(R[\F]\times\G) & \leq & \bigwedge_{\mathbb{D}\ni\D\leq\G}\left(\lambda_{X}(\F)(x)\vee\lambda_{Z}(\D)(z)\right)\\
 & \leq & \lambda_{X}(\F)(x)\vee\lambda_{Z}(\G)(z)\mbox{, because \ensuremath{Z\mbox{ is \ensuremath{\mathbb{D}}-based}.}}
\end{eqnarray*}

$(2\Rightarrow3)$ Is trivial.

$(3\Rightarrow1)$ Let $\F\in\mathbb{F}X$ and $x\in X$. We want
to show that 
\[
c_{\mathbb{D}}^{R(x)}R[\F]\leq\alpha:=\lambda_{X}(\F)(x).
\]
Since $R\times Id_{Z}$ is $\mathbb{F}_{0}$-compact for every topological
$\mathbb{D}$-based atomic approach space $Z$, then for every such
$(Z,\lambda_{Z})$ with non-isolated point $\infty$, 
\[
c_{\mathbb{F}_{0}}^{R(x)\times\{\infty\}}(R[\F]\times\N(\infty))\leq\alpha\vee0=\alpha,
\]
and Theorem \ref{thm:mainproduct} applies to the effect that that
$c_{\mathbb{D}}^{R(x)}R[\F]\leq\alpha$. \end{proof}
\begin{rem}[on infinite products]
 \cite[Theorem 14]{mynardmeasureCAP} provides a Tychonoff Theorem
for the general measure $c_{\mathbb{F}}^{\A}(\F)$ as defined in (\ref{eq:Dcompmeasure}).
However, there is an obvious error in the proof tantamount to writing
that 
\[
\bigvee_{i\in I}a_{i}+\bigvee_{i\in I}b_{i}=\bigvee_{i\in I}(a_{i}+b_{i})
\]
which is obviously false. This problem disappears when $\A$ is restricted
to a family of subsets, rather than functions. In this case, $\bigvee_{i\in I}\A_{i}$
becomes $\prod_{i\in I}\A_{i}$. Thus \cite[Theorem 14]{mynardmeasureCAP}
should read:\end{rem}
\begin{thm}
Let $(X_{i},\lambda_{i})_{i\in I}$ be a family of convergence approach
spaces, let $\A_{i}\subset2^{X_{i}}$, and let $\F$ be a filter on
$\prod_{i\in I}X_{i}$. Then
\[
c^{\prod_{i\in I}\A_{i}}(\F)=\bigvee_{i\in I}c^{\A_{i}}(p_{i}[\F]),
\]
where $p_{i}:\prod_{i\in I}X_{i}\to X_{i}$ is the $i^{th}$-projection.
\end{thm}

\section{Applications}

Taking $\F=\{X\}$, $A=X$, $\G=\{Y\}$, $B=Y$ in Theorem \ref{thm:mainproduct},
we obtain an extension to $\mathbf{Cap}$ in terms of measure of compactness
of the classical topological fact that a product of a compact space
with a space that is respectively compact, countably compact, or Lindel\"of
is also compact, countably compact, or Lindel\"of, respectively:
\begin{cor}
Let $\mathbb{D}$ be a composable class of filters containing principal
filters, and let $X$ and $Y$ be two convergence approach spaces.
\[
c_{\mathbb{D}}(X\times Y)\leq c_{\mathbb{D}}X\vee c_{\mathbb{F}}Y.
\]

\end{cor}
In particular, (for $\mathbb{D}=\mathbb{F}_{1}$) the measure of countable
compactness (in the sense of \cite{baekelandlowen}) of a product
of two spaces is not larger than the supremum of the measure of countable
compactness and measure of compactness of the factors. Similarly,
for $\mathbb{D}=\mathbb{F}_{\wedge1}$, the measure of Linde\"of
(in the sense of \cite{baekelandlowen.LS}) of a product of two spaces
is not larger than the supremum of the measure of Linde\"of and measure
of compactness of the factors. 

Both instances appear to be new, even if they are probably part of
the folklore on approach spaces.

On the other hand, applying Theorem \ref{thm:mainproduct} with $\F=\{X\}$,
$\alpha=0$, and $A=X$, yields the following generalization of the
Kuratowski-Mr\`owka characterization of compactness:
\begin{cor}
Let $\mathbb{D}$ be a composable class of filters containing principal
filters, and let $X$ be a convergence approach space. Then the following
are equivalent:
\begin{enumerate}
\item $X$ is $\mathbb{D}$-compact; 
\item For every $\mathbb{D}$-based convergence approach space $(Y,\lambda_{Y}),\, p_{Y}:X\times Y\to Y$
is $\mathbb{D}$-perfect; 
\item For every atomic $\mathbb{D}$-based atomic topological (approach)
space $Y,\, p_{Y}:X\times Y\to Y$ is closed. 
\end{enumerate}
\end{cor}
\begin{proof}
$\left(1\then2\right):$ To see that $p_{Y}^{-}:Y\rightrightarrows X\times Y$
is $\mathbb{D}$-compact, we need to show that 
\[
c_{\mathbb{D}}^{X\times\{y\}}(X\times\G)\leq\lambda_{Y}(\G)(y)
\]
for any $\G\in\mathbb{F}Y$ and $y\in Y$. For each $\D\in\mathbb{D}$
with $\D\leq\G$, Theorem \ref{thm:mainproduct} applies for $\F=\{X\}$,
$A=\{X\}$, $\G=\D$, $B=\{y\}$ and $\alpha=c_{\mathbb{D}}^{X}X=0$
to the effect that 
\[
c_{\mathbb{D}}^{X\times\{y\}}(X\times\G)\leq c_{\mathbb{D}}^{X\times\{y\}}(X\times\D)\leq c_{\mathbb{F}}^{\{y\}}(\D)\leq\lambda_{Y}(\D)(y)
\]
so that 
\[
c_{\mathbb{D}}^{X\times\{y\}}(X\times\G)\leq\bigwedge_{\mathbb{D\ni\D\leq\G}}\lambda_{Y}(\D)(y)=\lambda_{Y}(\G)(y).
\]

$\left(2\then3\right)$ is clear, and $(3\then1)$ because $(3)$
means that for every $\mathbb{D}$-based atomic topological approach
space $Y$, with non-isolated point $\infty,$ 
\[
c_{\mathbb{F}_{0}}^{A\times\{\infty\}}(X\times\N(\infty))\leq c_{\mathbb{F}}^{\{\infty\}}(\N(\infty))=0
\]
 so that Theorem \ref{thm:mainproduct} applies to the effect that
$c_{\mathbb{D}}^{X}X=0$.
\end{proof}
In particular, when $\mathbb{D}$ ranges over the classes $\mathbb{F}$,
$\mathbb{F}_{1}$ and $\mathbb{F}_{\wedge1}$ respectively, we obtain
the instances below. By analogy with the case of topological spaces,
we call \emph{first-countable }an $\mathbb{F}_{1}$-based convergence
approach space, and a $P$-\emph{convergence approach space }one that
is $\mathbb{F}_{\wedge1}$-based. On the other hand, we call an $\mathbb{F}_{0}$-based
convergence approach space \emph{finitely generated}, because a pre-approach
space is finitely generated in the sense of \cite{Lowen97} if and
only if it is finitely generated in this sense.

\begin{cor}
Let $(X,\lambda_{X})$ be a convergence approach space. Then the following
are equivalent:
\begin{enumerate}
\item $X$ is compact; 
\item For every convergence approach space $(Y,\lambda_{Y}),\, p_{Y}:X\times Y\to Y$
is perfect; 
\item For every atomic topological (approach) space $Y,\, p_{Y}:X\times Y\to Y$
is closed. 
\end{enumerate}
\end{cor}

\begin{cor}
Let $(X,\lambda_{X})$ be a convergence approach space. The following
are equivalent:
\begin{enumerate}
\item $X$ is countably compact; 
\item For every first-countable (%
\footnote{As noted in \cite{myn.applofcompact} in the context of $\mathbf{Conv}$,
we could extend this to spaces based in contours of countably based
filters, but this is of little importance here. In the case of topological
spaces, they turn out to be exactly subsequential spaces, that is,
subspaces of sequential spaces. See \cite{myn.applofcompact} for
details.%
}) convergence approach space $(Y,\lambda_{Y}),p_{Y}:X\times Y\to Y$
is countably perfect; 
\item For every first-countable atomic topological (approach) space $Y,\, p_{Y}:X\times Y\to Y$
is closed.
\end{enumerate}
\end{cor}

\begin{cor}
Let $(X,\lambda_{X})$ be a convergence approach space. The following
are equivalent:
\begin{enumerate}
\item $X$ is Lindel\"of; 
\item For every $P$-convergence approach space $(Y,\lambda_{Y}),p_{Y}:X\times Y\to Y$
is inversely Lindel\"of; 
\item For every atomic topological $P$-space $Y$ (seen as an approach
space) $p_{Y}:X\times Y\to Y$ is closed.
\end{enumerate}
\end{cor}
On the other hand, Theorem \ref{thm:mainformaps} combined with Proposition
\ref{prop:closedF0compact} readily gives:
\begin{cor}
\label{cor:Dperfectprod}Let $\mathbb{D}$ be a composable class of
filters containing principal filters. Let $(X,\lambda_{X})$ and $(Y,\lambda_{Y})$
be two convergence approach spaces, and let $f:X\to Y$. Then the
following are equivalent: 
\begin{enumerate}
\item $f$ is $\mathbb{D}$-perfect; 
\item For every $\mathbb{D}$-based convergence approach space $Z$, $f\times Id_{Z}:X\times Z\to Y\times Z$
is $\mathbb{D}$-perfect; 
\item For every atomic topological $\mathbb{D}$-based approach space $Z,$
$f\times Id_{Z}:X\times Z\to Y\times Z$ is closed. 
\end{enumerate}
\end{cor}
Similarly, Theorem \ref{thm:mainformaps} combines with Theorem \ref{th:Dquotient}
to the effect that:

\begin{cor}
\label{cor:Dquotientprod}Let $\mathbb{D}$ be a composable class
of filters containing principal filters. Let $(X,\lambda_{X})$ and
$(Y,\lambda_{Y})$ be two convergence approach spaces, and let $f:X\to Y$
be a surjective map. Then the following are equivalent: 
\begin{enumerate}
\item $f$ is $\mathbb{D}$-quotient; 
\item For every $\mathbb{D}$-based convergence approach space $Z$, $f\times Id_{Z}:X\times Z\to Y\times Z$
is $\mathbb{D}$-quotient; 
\item For every atomic topological $\mathbb{D}$-based approach space $Z,$
$f\times Id_{Z}:X\times Z\to Y\times Z$ is hereditarily quotient. 
\end{enumerate}
\end{cor}

In particular, when $\mathbb{D=F}$ is the class of all filters, we
obtain:

\begin{cor}
\label{cor:biquotient} Let $(X,\lambda_{X})$ and $(Y,\lambda_{Y})$
be two convergence approach spaces, and let $f:X\to Y$ be a surjective
map. Then the following are equivalent: 
\begin{enumerate}
\item $f$ is biquotient; 
\item For every convergence approach space $Z$, $f\times Id_{Z}:X\times Z\to Y\times Z$
is biquotient; 
\item For every atomic topological approach space $Z,$ $f\times Id_{Z}:X\times Z\to Y\times Z$
is hereditarily quotient. 
\end{enumerate}
\end{cor}

\begin{cor}
\label{cor:perfectproper} Let $(X,\lambda_{X})$ and $(Y,\lambda_{Y})$
be two convergence approach spaces, and let $f:X\to Y$. Then the
following are equivalent: 
\begin{enumerate}
\item $f$ is perfect; 
\item For every convergence approach space $Z$, $f\times Id_{Z}:X\times Z\to Y\times Z$
is perfect; 
\item For every atomic topological approach space $Z,$ $f\times Id_{Z}:X\times Z\to Y\times Z$
is closed. 
\end{enumerate}
\end{cor}
In \cite{lowen.Mrowka}, Lowen and al. call a map $f:X\to Y$ between
two \emph{approach }spaces \emph{proper }if $f\times Id_{Z}$ is closed
for every approach space $Z$. In view of Corollary \ref{cor:perfectproper},
our perfect maps extend to $\mathbf{Cap}$ the concept of proper maps
of \cite{lowen.Mrowka}. Additionally, the equivalence between (1)
and (2) in \cite[Proposition 3.3]{lowen.Mrowka} states that a map
between two approach spaces is proper if and only if it is closed
and has compact fibers ($0$-compact in the terminology of \cite{lowen.Mrowka}).
Theorem \ref{thm:Dperfect} for $\mathbb{D=F}$ and Corollary \ref{cor:perfectproper}
recover this equivalence, and delineate the conditions of an extension
of this result to $\mathbf{Cap}$ (namely, $X$ needs to remain an
approach space, but $Y$ can be an arbitrary convergence approach
space). At any rate, the proper (no punn intended) notion yielding
a characterization of maps whose product with every identity map is
closed (in $\mathbf{Cap}$ and not only $\mathbf{Ap}$) appears to
be that of perfect maps, which ultimately depends on that of compact
relation. That the condition reduces to the closedness of the map
and compactness of the fibers is specific to $\mathbf{Ap}$, as shows
Theorem \ref{thm:Dperfect}. 

Maybe more importantly, the viewpoint in terms of $\mathbb{D}$-compact
relations unveils the relationships between similar characterizations
in terms of products of variants of perfects maps on one hand (Corollary
\ref{cor:Dperfectprod}) and variants of quotient maps on the other
hand (Corollary \ref{cor:Dquotientprod}) as two instances of the
same result. While this was already observed in \cite{myn.applofcompact}
in the concept of $\mathbf{Conv}$, it is remarkable that this turns
out to extend fully to $\mathbf{Cap}$.

On the other hand, letting $\mathbb{D}$ range over other classes
($\mathbb{F}_{1}$, $\mathbb{F}_{\wedge1}$, $\mathbb{F}_{0}$) yields
other variants of Corollaries \ref{cor:biquotient} and \ref{cor:perfectproper}:

\begin{cor}
\label{cor:countbiquotient} Let $(X,\lambda_{X})$ and $(Y,\lambda_{Y})$
be two convergence approach spaces, and let $f:X\to Y$ be a surjective
map. Then the following are equivalent: 
\begin{enumerate}
\item $f$ is countably biquotient; 
\item For every first-countable convergence approach space $Z$, $f\times Id_{Z}:X\times Z\to Y\times Z$
is countably biquotient; 
\item For every atomic first-countable topological approach space $Z,$
$f\times Id_{Z}:X\times Z\to Y\times Z$ is hereditarily quotient. 
\end{enumerate}
\end{cor}

\begin{cor}
Let $(X,\lambda_{X})$ and $(Y,\lambda_{Y})$ be two convergence approach
spaces, and let $f:X\to Y$. Then the following are equivalent: 
\begin{enumerate}
\item $f$ is countably perfect; 
\item For every first-countable convergence approach space $Z$, $f\times Id_{Z}:X\times Z\to Y\times Z$
is countably perfect; 
\item For every atomic first-countable topological approach space $Z,$
$f\times Id_{Z}:X\times Z\to Y\times Z$ is closed. 
\end{enumerate}
\end{cor}

\begin{cor}
\label{cor:weaklybiq} Let $(X,\lambda_{X})$ and $(Y,\lambda_{Y})$
be two convergence approach spaces, and let $f:X\to Y$ be a surjective
map. Then the following are equivalent: 
\begin{enumerate}
\item $f$ is weakly biquotient; 
\item For every $P$-convergence approach space $Z$, $f\times Id_{Z}:X\times Z\to Y\times Z$
is weakly biquotient; 
\item For every atomic topological approach $P$-space $Z,$ $f\times Id_{Z}:X\times Z\to Y\times Z$
is hereditarily quotient. 
\end{enumerate}
\end{cor}

\begin{cor}
Let $(X,\lambda_{X})$ and $(Y,\lambda_{Y})$ be two convergence approach
spaces, and let $f:X\to Y$. Then the following are equivalent: 
\begin{enumerate}
\item $f$ is inversely Lindel\"of; 
\item For every $P$-convergence approach space $Z$, $f\times Id_{Z}:X\times Z\to Y\times Z$
is inversely Lindel\"of; 
\item For every atomic topological approach $P$-space $Z,$ $f\times Id_{Z}:X\times Z\to Y\times Z$
is closed. 
\end{enumerate}
\end{cor}

\begin{cor}
\label{cor:weaklybiq-1} Let $(X,\lambda_{X})$ and $(Y,\lambda_{Y})$
be two convergence approach spaces, and let $f:X\to Y$ be a surjective
map. Then the following are equivalent: 
\begin{enumerate}
\item $f$ is hereditarily biquotient; 
\item For every finitely generated ocnvergence approach space $Z$, $f\times Id_{Z}:X\times Z\to Y\times Z$
is hereditarily quotient;
\item For every atomic topological finitely generated approach space $Z,$
$f\times Id_{Z}:X\times Z\to Y\times Z$ is hereditarily quotient. 
\end{enumerate}
\end{cor}

\begin{cor}
Let $(X,\lambda_{X})$ and $(Y,\lambda_{Y})$ be two convergence approach
spaces, and let $f:X\to Y$. Then the following are equivalent: 
\begin{enumerate}
\item $f$ is closed; 
\item For every finitely generated convergence approach space $Z$, $f\times Id_{Z}:X\times Z\to Y\times Z$
is closed; 
\item For every atomic finitely generated topological approach space $Z,$
$f\times Id_{Z}:X\times Z\to Y\times Z$ is closed. 
\end{enumerate}
\end{cor}
Finally, let us note that applying Theorem \ref{thm:mainproduct}
for $A=\{x\}$ and $B=\{y\}$, yields, via Theorem \ref{th:AdhJarecompact},
the following extension to $\mathbf{Cap}$ of \cite[Theorem 8]{myn.applofcompact}:
\begin{cor}
\label{cor:reflector}Let $\mathbb{D}$ be a composable class of filters
containing principal filters. Let $(X,\lambda_{X})$ be a convergence
approach space and let $\lambda_{2}$ be another convergence approach
structure on $X$. The following are equivalent:
\begin{enumerate}
\item $\lambda_{2}\geq\Adh_{\mathbb{D}}\lambda_{X}$;
\item For every $\mathbb{D}$-based convergence approach space $(Y,\lambda_{Y})$,
\[
\Adh_{\mathbb{D}}(\lambda_{X}\times\lambda_{Y})\leq\lambda_{2}\times\Adh_{\mathbb{F}}\lambda_{Y}
\]

\item For every $\mathbb{D}$-based atomic topological approach space $(Y,\lambda_{Y})$,
\[
\Adh_{\mathbb{F}_{0}}(\lambda_{X}\times\lambda_{Y})\leq\lambda_{2}\times\lambda_{Y}.
\]

\end{enumerate}
\end{cor}
The significance of this type of results appears fully in the context
of \emph{modified duality }as developed in \cite{mynard,Mynard.survey}.
For instance, when $\mathbb{D}$ is the class of all filters, $(1\then2)$
simply shows that the reflector on pseudo-approach spaces $\Adh_{\mathbb{F}}$
communtes with (finite) products. As a result $\mathbf{Psap}$ is
cartesian-closed. More importantly, since $\Adh_{\mathbb{F}_{0}}$
is the projector on $\mathbf{Prap}$, $(3\then1)$ shows (see \cite{mynard,Mynard.survey}
for details) that $\mathbf{Psap}$ is the cartesian-closed hull of
$\mathbf{Prap}$, which is \cite[Theorem 5.9]{lowe88}. See the aforementioned
references for details and other applications of results akin to Corollary
\ref{cor:reflector}.

\bibliographystyle{plain}

\end{document}